\documentclass[review,onefignum,onetabnum]{siam}

\usepackage{booktabs}
\usepackage{amsmath}
\usepackage{chngcntr}
\usepackage{amssymb}
\usepackage{graphicx}
\usepackage{microtype}
\usepackage{float}
\usepackage{theorem}
\usepackage{overpic}
\usepackage{multirow}
\usepackage{etoolbox}

\newtheorem{assumption}{Assumption}[section]
\newtheorem{remark}{Remark}[section]
\makeatletter

\expandafter\let\csname algorithm*\endcsname\relax
\expandafter\let\csname endalgorithm*\endcsname\relax
\makeatother

\usepackage[ruled,linesnumbered]{algorithm2e}

\usepackage{chngcntr}          
\counterwithin{algocf}{section} 

\makeatletter
\@ifpackageloaded{hyperref}{\hypersetup{hidelinks}}{\usepackage[hidelinks]{hyperref}}
\makeatother
\usepackage{cleveref}

\allowdisplaybreaks[2]
\SetAlgoSkip{smallskip}

\title{Nonlinear parametrization solver for fractional Burgers equations}
\author{
Haojun Qin%
\thanks{Haojun Qin, Zhiwei Gao, and Jinye Shen contributed equally as co-first authors.}
\and
Zhiwei Gao\footnotemark[1]
\and
Jinye Shen\footnotemark[1]
\and
George E.~Karniadakis%
\thanks{Division of Applied Mathematics, Brown University.
Emails: \email{haojun\_qin@brown.edu} (H.~Qin), \email{zhiwei\_gao@brown.edu} (Z.~Gao),
\email{jinye\_shen@brown.edu} (J.~Shen), \email{george\_karniadakis@brown.edu} (G.~E.~Karniadakis).
Corresponding author: G.~E.~Karniadakis.}
}

\date{}
\begin{document}

\maketitle
\begin{abstract}
Fractional Burgers equations pose substantial challenges for classical numerical methods due to the combined effects of nonlocality and shock-forming nonlinear dynamics. In particular, linear approximation frameworks—such as spectral, finite-difference, or discontinuous Galerkin methods—often suffer from Gibbs-type oscillations or require carefully tuned stabilization mechanisms, whose effectiveness degrades in transport-dominated and long-time integration regimes.
In this work, we introduce a sequential-in-time nonlinear parametrization (STNP) for solving fractional Burgers equations, including models with a fractional Laplacian or with nonlocal nonlinear fluxes. The solution is represented by a nonlinear parametric ansatz, and the parameter evolution is obtained by projecting the governing dynamics onto the tangent space of the parameter manifold through a regularized least-squares formulation at each time step. This yields a well-posed and stable time-marching scheme that preserves causality and avoids global-in-time optimization.
We provide a theoretical analysis of the resulting projected dynamics, including a stability estimate and an {\it a posteriori error} bound that explicitly decomposes the total error into contributions from initial condition fitting, projection residuals, and discretization of fractional operators. Our analysis clarifies the stabilizing role of regularization and quantifies its interaction with the nonlocal discretization error.
Numerical experiments for both fractional Burgers models demonstrate that STNP achieves oscillation-free shock resolution and accurately captures long-time dynamics. The method consistently outperforms high-order spectral schemes augmented with spectral vanishing viscosity, while requiring significantly fewer degrees of freedom and avoiding ad hoc stabilization. 
\end{abstract}

\section{Introduction}
Fractional partial differential equations (FPDEs) have shown great potential in modeling nonlocal behavior of diverse systems, such as anomalous diffusion in porous media \cite{albinali2016analytical} and viscoelastic materials \cite{fukunaga2015fractional}. In particular, as a generalization of the integer-order Burgers equation, the spatial fractional Burgers equation  has attracted considerable interest due to its applications in the areas of over-driven detonation in gases, nonlinear Markov process propagation of chaos, etc. \cite{Shen2022,mao2017fractional}. 

Specifically, there are two major extensions of the classical Burgers equations. Firstly, the fractional Laplacian operator, i.e., $(-\Delta)^{\frac{\alpha}{2}}$, is used to replace the standard diffusion term, leading to
 \begin{equation}\label{FBEFL}
     u_t + \frac{1}{2}(u^2)_x = -\varepsilon (-\Delta)^{\frac{\alpha}{2}} u, \quad 1 < \alpha < 2,
 \end{equation}
 which is denoted as fractional Burgers equation with fractional Laplacian (FBEFL). Here, the nonlocal operator is associated with the linear term and thus serves as a super-diffusion term, smoothing out the solution. For the case of $\alpha >1$, the existence and uniqueness of a regular solution have been established
 in \cite{alibaud2007occurrence,kiselev2008blow} .
 Another extension, called the fractional Burgers equation with nonlocal nonlinearity (FBENN), is obtained by replacing the nonlinear term with the composition of two left Caputo derivatives, hence  representing a nonlocal generalization of the classic Burgers equation, i.e.,
 \begin{equation}\label{FBENN}
     u_t + \frac{1}{2} {}_a^C D^{\beta} ({}_a^C D^{1-\beta}u)^2 = \varepsilon u_{xx},
 \end{equation}
 where $0 \leq \beta \leq 1$. With different choices of $\beta$, \eqref{FBENN} can recover the classic Burgers equation or a nonlinear diffusion equation.  More importantly, by adopting a fractional generalization of the Hopf–Cole transformation, the solution can be obtained explicitly \cite{mivskinis2013generalization}, i.e.,  
\begin{equation}\label{exact solution}
    u(x,t) = -2\varepsilon \,{}_a^C D^\beta \left[  \log \left( \frac{1}{\sqrt{4\pi \varepsilon t}} \int_{\mathbb{R}} e^{- \frac{|x-y|^2}{4\varepsilon t} } e^{-\frac{1}{2\varepsilon} {}_a I^\beta u_0(y)} \,\mathrm{d}y \right) \right].
\end{equation}
The notation and some other properties will be introduced later in section \ref{section2}. 

For the classical Burgers equation, shock waves~\cite{leveque1992numerical} can develop even from smooth initial data, which in turn induces Gibbs phenomena~\cite{lax2006gibbs} and spurious oscillations in standard \emph{linear parametrization} schemes, such as spectral expansions, finite differences (FD), and discontinuous Galerkin (DG) discretizations. In the integer-order setting, these issues are commonly mitigated by local shock-capturing techniques, e.g., upwind fluxes~\cite{courant1952solution} and ENO/WENO reconstructions~\cite{shu1999high}, which effectively stabilize the underlying linear representations by adding localized dissipation or nonlinear reconstruction.

For spatial fractional Burgers equations, the difficulty of linear parametrizations becomes more pronounced due to the intrinsic nonlocality of fractional operators. The fractional Laplacian and Caputo-type derivatives introduce global coupling, so that the evolution of each linear degree of freedom is influenced by the solution over the entire domain. As a consequence, linear discretizations often lead to denseand globally coupled operators. Moreover,  local stabilization mechanisms do not directly transfer since  numerical perturbations generated near shocks can spread nonlocally, while adding dissipation to suppress oscillations may easily degrade accuracy.

Several approaches can be viewed as stabilized linear parametrizations for fractional models. Droniou~\cite{droniou2010numerical} analyzed a broad class of FD schemes within a rigorous convergence framework. Xu and Hesthaven~\cite{xu2014discontinuous} extended the LDG methodology to fractional convection--diffusion equations, maintaining a DG-type linear representation while incorporating the fractional operator. More recently, Mao and Karniadakis~\cite{mao2017fractional} introduced spectral vanishing viscosity (SVV) into spectral and spectral-element LDG approximations to control Gibbs-type oscillations while preserving high-order accuracy~\cite{maday1989analysis}. While successful, these stabilized linear parametrizations can still be costly in transport-dominated regimes. Resolving sharp layers typically requires many degrees of freedom, and the stabilization parameters often need careful tuning to balance oscillation control against shock smearing and loss of accuracy, especially over long-time integration.

Motivated by these limitations of linear parametrizations, we propose a sequential-in-time nonlinear parametrization (STNP) as an efficient solver for two fractional Burgers equations, FBEFL and FBENN. STNP represents the solution by a parametric ansatz \(\hat{u}(x,t)=\Phi(x;q(t))\) and evolves the parameters through a tangent-space projection of the PDE dynamics. At each time step, \(\dot q(t)\) is obtained from a regularized least-squares problem, followed by standard time integration. This yields a well-posed update and provides a stable mechanism to control error growth induced by discretizing the nonlocal terms. We further establish stability results and derive {\it a posteriori} error estimates for the resulting projected dynamics. Our numerical experiments demonstrate that STNP can accurately and efficiently solve both FBEFL and FBENN while resolving shocks without spurious oscillations.

From a numerical analysis perspective, STNP can be interpreted as a nonlinear Galerkin-type method, in which the solution evolves on a low-dimensional nonlinear manifold and the governing dynamics are enforced through a tangent-space projection. 
A key feature of the proposed framework is the introduction of a regularized tangent-space projection, which plays a dual role in stabilizing the projected dynamics and enabling a rigorous {\it a posteriori}  error analysis.
In summary, our new contributions in this paper are:
\begin{itemize} 
\item Introduction of STNP as a nonlinear solver (not a learning framework).
\item Stability bound for regularized tangent projection.
\item {\it A posteriori} error decomposition (initial fit / projection / discretization).
\item Numerical evidence on fractional Burgers with shocks.
\end{itemize}

The remainder of this paper is organized as follows. In Section~\ref{section2} we review the basic definitions of fractional derivatives and the numerical discretizations used in this work. In Section~\ref{section3} we present regularity results for FBEFL and FBENN. In Section~\ref{section4} we introduce the STNP algorithm and its implementation details. In Section~\ref{section5} we justify the necessity of regularization and derive stability and {\it a posteriori} error estimates. In Section~\ref{section6} we provide numerical experiments that validate the analysis and demonstrate the performance of the proposed method.

\section{Background}\label{section2}
In this section, we introduce the notations of fractional calculus, including the Riemann--Liouville derivative, the Caputo derivative, and the fractional Laplacian. We then describe the corresponding numerical discretizations.

\subsection{Fractional calculus}
First, the left Riemann--Liouville integral of order $\beta$ is defined as
\begin{equation}
    ({}_aI_x^{\beta} u)(x) = \int_{a}^x \omega_{\beta}(x-s)u(s)\,\mathrm{d}s,
\end{equation}
where $\omega_{\beta}(x) = \frac{x^{\beta-1}}{\Gamma(\beta)}$ for $\beta > 0$ and $\Gamma(x)$ denotes the standard Gamma function. The right Riemann--Liouville integral of order $\beta$ is defined as 
\begin{equation}
    ({}_xI_a^{\beta} u)(x) = \int_{x}^a \omega_{\beta}(s-x)u(s)\,\mathrm{d}s.
\end{equation}
And the left Caputo derivative of order $\alpha \in (0,1)$ is defined as, 
\begin{equation}
    ({}^C_aD_x^{\alpha} u)(x) = ({}_a I_x^{1-\alpha} u')(x) = \frac{1}{\Gamma(1-\alpha)}\int_a^x \frac{u'(s)}{(x-s)^{\alpha}}\,\mathrm{d}s. 
\end{equation}
Similarly, the left Riemann--Liouville derivative of order $\alpha \in (0,1)$ is defined as,
\begin{equation}
    ({}^{RL}_a D_x^{\alpha} u)(x) = ({}_a I_x^{1-\alpha} u)'(x) = \frac{1}{\Gamma(1-\alpha)} \frac{\mathrm{d}}{\mathrm{d}x} \int_{a}^x \frac{u(s)}{(x-s)^{\alpha}}\,\mathrm{d}s.
\end{equation}
And the right Riemann--Liouville derivative of order $\alpha \in (0,1)$ is 
\begin{equation}
    ({}^{RL}_x D_a^{\alpha} u)(x) = ({}_x I_a^{1-\alpha} u)'(x) = \frac{1}{\Gamma(1-\alpha)} \left(-\frac{\mathrm{d}}{\mathrm{d}x} \right)\int_{x}^a \frac{u(s)}{(s-x)^{\alpha}}\,\mathrm{d}s.
\end{equation}
Based on these definitions, we can define the fractional Laplaician as a sum of two Riemann--Liouvile derivatives, i.e., 

\begin{equation}\label{Riesz2}
    (-\Delta)^{\frac{\alpha} 2} u(x) = \frac{1}{2\cos\left( \frac{\alpha \pi}{2}\right)} \left( {}_{-\infty}^{RL}D_x^{\alpha} u(x) + {}_{x}^{RL}D_{+\infty}^{\alpha}u(x) \right).
\end{equation}
It is worth mentioning that the fractional Laplacian can be defined in several equivalent ways \cite{lischke2020fractional}. In this paper, we restrict ourselves to the definition based on Riemann--Liouville derivatives since it is more convenient to construct finite difference schemes.
When $\alpha = 2$, the fractional Laplace operator is consistent with the classic Laplace operator. To discretize the fractional Laplacian, it suffices to discretize the Riemann--Liouville derivative. In this work, we employ the shifted Gr\"unwald--Letnikov scheme for the discretization of the Riemann--Liouville derivative, which is introduced in the following subsection.

\subsection{Discretization of the fractional derivatives}
\label{section_2.2}
Since fractional derivatives are nonlocal and depend on values over the entire domain, they cannot be discretized using standard local finite-difference schemes. For the Caputo derivative ${}^C_aD^\alpha$ with $\alpha \in (0,1)$, defined on $[a,b]$, one of the most commonly used methods is the L1 scheme, which is based on piecewise linear interpolation \cite{sun2006fully}. For a fixed $N$, let $h = \frac{b-a}{N}$. The grid points are $x_k = a + k h$, for $k = 0,1,2,\ldots,N$. The Caputo derivative at $x_n$ is approximated by
\begin{equation}\label{L1}
    {}_a^C D^\alpha u(x)\big|_{x = x_n} \approx \frac{h^{-\alpha}}{\Gamma(2- \alpha)} \left[ a^{(\alpha)}_0 u(x_n) - \sum_{k=1}^{n-1}\big(a^{(\alpha)}_{n-k-1} - a^{(\alpha)}_{n-k}\big)u(x_k) - a^{(\alpha)}_{n-1}u(x_0)\right],
\end{equation}
with
\begin{equation}
    a^{(\alpha)}_{l} = (l+1)^{1-\alpha} - l^{1-\alpha}.
\end{equation}
The truncation error of the L1 scheme is $\mathcal{O}(h^{2 - \alpha})$ \cite{sun2006fully}.

For the fractional Laplacian, since it can be expressed as a sum of two Riemann--Liouville derivatives, we adopt the shifted Gr\"unwald--Letnikov formula to approximate it. First, define
\begin{equation}
    A^\alpha_{h,p} u(x) = h^{-\alpha} \sum_{k=0}^\infty g^{(\alpha)}_k \, u\!\big(x-(k-p)h\big),
\end{equation}
where $p$ denotes the shift, and
\begin{equation}
    g^{(\alpha)}_k = (-1)^k \binom{\alpha}{k}, \quad \binom{\alpha}{k} =  \frac{\alpha(\alpha-1)\cdots(\alpha-k+1)}{k!}.
\end{equation}
Then the first-order approximation \cite{tian2015class} is given by
\begin{equation}\label{GL1}
    {}_{-\infty}^{RL}D^{\alpha}_{x}u(x) = A^{\alpha}_{h,p}u(x) + \mathcal{O}(h).
\end{equation}
Similarly, for $p\neq q$, the second-order approximation is given by \cite{tian2015class}
\begin{equation}\label{GL2}
    {}_{-\infty}^{RL}D^{\alpha}_{x}u(x) = \lambda_1 A^{\alpha}_{h,p} u(x) + \lambda_2 A^\alpha_{h,q}u(x) + \mathcal{O}(h^2),
\end{equation}
where
\begin{equation}
    \lambda_1 = \frac{\alpha - 2q}{2(p-q)}, \qquad \lambda_2 = \frac{2p-\alpha}{2(p-q)}.
\end{equation}
For other higher-order schemes, we refer the reader to \cite{hao2015fourth,zhou2013quasi}. We emphasize that, in contrast to local finite-difference schemes, the above fractional discretizations lead to convolution-type operators. In practice, the infinite series in $A^\alpha_{h,p}$ is truncated on a finite grid, and the resulting discrete operator can be viewed as a dense matrix acting on the grid values of $u$. The weights $\{g_k^{(\alpha)}\}$ decay algebraically and can be generated efficiently via the recursion
$g_0^{(\alpha)}=1$ and $g_k^{(\alpha)}=\left(1-\frac{\alpha+1}{k}\right)g_{k-1}^{(\alpha)}$,
so that the discretization can be implemented without explicitly forming the dense matrix. This highlights the key nonlocality of fractional derivatives and motivates the use of shifted Gr\"unwald--Letnikov formulas in our numerical scheme.

\section{Fractional Burgers equations}
\label{section3}
In this part, we recall the definitions of FBEFL and FBENN and then discuss their theoretical properties, including regularity, asymptotic behavior, and connections with other equations.

\subsection{FBEFL}

The Burgers equation with fractional Laplacian is defined as follows,
\begin{equation}\label{FBEFL2}
    \begin{cases}
        u_t + \frac{1}{2}(u^2)_x = -\varepsilon (-\Delta)^{\frac{\alpha}{2}} u & \text{in} \;\mathbb{R} \times \mathbb{R}^+, \\
        u(x,0) = u_0(x) & \text{in} \;\mathbb{R},
    \end{cases}
\end{equation}
where $1 < \alpha < 2$ and $\varepsilon>0$ is the diffusion rate.  It has been proved that for the periodic initial data, the case $ 1 < \alpha \leq 2$ is subcritical, where the real analytic solutions exist globally. The proof is similar to the classic case
 $\alpha =2$. Here we cite the result in \cite{kiselev2008blow}:
\begin{theorem} \cite{kiselev2008blow}
Assume that $1< \alpha \leq 2$, and the initial data $u_0(x) \in H^s(\mathbb{S}^1)$, for $s >
 \max(\frac{3}{2} - 2\alpha,0)$. Then there exists a unique global solution $u(x,t)$ of (\ref{FBEFL2}) such that $u$
 belongs to $C([0,\infty),H^s(\mathbb{S}^1))$. Moreover, $u(x,t)$ is real analytic in $\mathbb{R}$ for $t > 0$.
\end{theorem}
For more general initial data $u_0 \in H^1(\mathbb{R})$, Biler \cite{biler1998fractal} has shown that for $\frac{3}{2} < \alpha \leq 2$, the solution still belongs to $H^1(\mathbb{R})$. The result in \cite{biler1998fractal} is as follows:
\begin{theorem}\cite{biler1998fractal}
    Assume that $\frac{3}{2} < \alpha \leq 2$, $T >0$ and the initial data $u_0 \in H^1(\mathbb{R})$, the solution of (\ref{FBEFL2}) is unique, i.e.,
    \begin{equation}
        u \in L^\infty((0,T);H^1(\mathbb{R})) \cap L^2((0,T); H^{1+\frac{\alpha}{2}}(\mathbb{R})),
    \end{equation}
    for each $T > 0$. And the solution decays so that 
    \begin{equation}
        \lim_{t\to \infty} \|u(\cdot,t)\|_{L^\infty} = 0.
    \end{equation}
\end{theorem}
The proof is the standard procedure of energy estimate in parabolic equations augmented with the Sobolev embedding.

\subsection{FBENN}
The fractional Burgers equation with nonlocal nonlinearity is defined as follows,
\begin{equation}
    \begin{cases}
        u_t + \frac{1}{2} {}^C_a D^{\beta}({}^C_a D^{1-\beta}u)^2 = \varepsilon u_{xx} & \text{in} \;\mathbb{R} \times \mathbb{R}^+, \\
        u(x,0) = u_0(x) & \text{in} \;\mathbb{R},
    \end{cases}
    \label{FBENN3}
\end{equation}
where $0 \leq \beta \leq 1$. 
Here ${}^C_a D^\alpha$ denotes the Caputo fractional derivative of order $\alpha$ with respect to the spatial variable $x$, and the parameter $\beta\in[0,1]$ interpolates between local and nonlocal convection mechanisms. In particular, the nonlinear flux is constructed through the fractional integral--derivative pair ${}^C_aD^\beta({}^C_aD^{1-\beta}u)^2$, where ${}^C_aD^{1-\beta}$ introduces a nonlocal averaging of $u$ before squaring, and ${}^C_aD^\beta$ further contributes a nonlocal differentiation to the resulting flux. The limiting cases recover familiar models: when $\beta=1$, the convective term reduces to the classical Burgers nonlinearity $\frac12 (u^2)_x$, whereas when $\beta=0$ it becomes $\frac12(u_x)^2$, corresponding to a Hamilton--Jacobi–type nonlinearity. Throughout, $\varepsilon>0$ is the viscosity coefficient, and the diffusion term $\varepsilon u_{xx}$ provides regularization and ensures well-posedness of the initial-value problem.

For the two Burgers-type equations considered above, directly employing classical solvers based on linear function parameterizations—such as spectral methods or discontinuous Galerkin (DG) methods—can be nontrivial. 
From the perspective of solution regularity, Burgers dynamics may develop steep gradients and shock-like transition layers, under which high-order polynomial approximations are susceptible to Gibbs-type oscillations unless careful stabilization and limiting procedures are incorporated. 
Moreover, the non-locality introduced by the fractional terms further complicates the discretization and increases the computational cost.
In view of these challenges, we adopt a nonlinear parameterization—specifically, neural networks—to approximate the solution, which will be introduced in the following section.
While the framework applies to general nonlinear parametrizations, including tensor decompositions \cite{oseledets2011tensor} and radial basis expansions \cite{wendland2004scattered}, we focus on neural networks due to their expressive power and ease of differentiation.

\section{Methodology}\label{section4}
In this section, we first introduce the fundamentals of nonlinear parametric approximation. We then employ a sequential-in-time technique to derive the time evolution of the parameters, and finally describe how the resulting framework is applied to the two Burgers-type equations introduced above.
\subsection{Nonlinear parametric approximation}
In this part, we consider the parametrizations that depend nonlinearly on the parameters. Let $q$ denote the parameters, and $\Phi(\cdot;q)$ denote the parametric approximation. Let $\Phi'(\cdot;q)$ be the derivative of $\Phi$ with respect to $q$;  $\Phi$ being  nonlinear is equivalent to $\Phi''(\cdot;q) \neq 0$. Examples of nonlinear approximations are deep neural network \cite{goodfellow2016deep}, tensor-based decompositions \cite{orus2014practical, oseledets2011tensor} and Multi-Gaussian approximations \cite{richings2015quantum}. 

We now describe how nonlinear parameterizations can be used as a numerical solver for time-dependent PDEs. 
The key idea is to represent the solution by a parametric ansatz $\hat{u}(x,t)=\Phi(x;q(t))$, and to evolve the parameters $q(t)$ in time so that the ansatz approximately satisfies the governing equation. 
Instead of minimizing a global-in-time objective over the entire space--time domain, we enforce the PDE locally in time: at each time level, we choose an update of $q(t)$ that best matches the PDE right-hand side in a least-squares sense. 
This results in a sequential procedure in which one repeatedly solves a small regression problem for $\dot q(t)$  and then advances $q$ using a standard time integrator.

\subsection{Sequential-in-time training}
Compared with the global-in-time training used in Physics-Informed Neural Networks (PINNs) and their variants \cite{pang2019fpinns, raissi2019physics}, we will discuss the sequential-in-time training method in this section, which treats time variables separately in order to preserve causality and reduce the number of parameters \cite{koch2007dynamical,bruna2024neural}.

Consider the following time-dependent partial differential equation,
\begin{equation}
    \frac{\partial }{\partial t}u(x,t) = f(t,x,u,Du), \quad \text{with} \; u(x,0) = u_0(x), 
\end{equation}
where $x \in \Omega \subset \mathbb{R}$, $u:\mathbb{R} \times \mathbb{R}^+ \to \mathbb{R}$ is a scalar function on both time and space, and $f$ is the differential operator which depends on $t$, $x$, $u$ and more generally, any type of fractional derivative of $u$. Then, we seek a nonlinear parametric approximation 
\begin{equation}
    u(x,t) \approx \hat{u}(x,t)= \Phi(x;q(t)),
\end{equation}
where $q(t)\in \mathcal{Q} \subseteq \mathbb{R}^{N_P}$ represents the time-dependent parameters. 

Now, we start to introduce the formulation of sequential-in-time training. 
By the chain rule, the time evolution of the numerical solution is defined by
\begin{equation}
    \frac{\partial \hat{u}}{\partial t} = \frac{\partial \hat{u}}{\partial q} \dot{q}.
\end{equation}
At each time step, the time derivative of the parameters, is obtained by solving the following least-squares problem,
\begin{equation}
    \frac{\partial q}{\partial t} = \operatorname{argmin}_{\gamma \in \mathcal{Q}} \mathcal{J}(\gamma),
\end{equation}
where the cost function is 
\begin{equation}\label{Leastsquare}
    \mathcal{J}(\gamma) = \left\|\frac{\partial \hat{u}}{\partial q}\gamma - f(t,\cdot,\hat{u},D\hat{u}) \right\|_{L^2}^2.
\end{equation}
The first-order optimality condition is 
\begin{equation}
    \nabla_\gamma \mathcal{J}(\gamma^*) = \left(\int_\Omega  \frac{\partial \hat{u}}{\partial q} \otimes \frac{\partial \hat{u}}{\partial q}   \,\mathrm{d}x \right) \gamma^* - \left(  \int_{\Omega} \frac{\partial \hat{u}}{\partial q} \otimes  f(t,x,\hat{u},D\hat{u}) \,\mathrm{d}x \right) = 0,
\end{equation}
where $\otimes$ denotes the outer product of the gradient with respect to the parameters. In practice, we choose some collocation points $\{x_1,x_2,...,x_{N_C}\} \subseteq \Omega$. These  collocation points induce a discrete least-squares approximation of the continuous least-squares problem \eqref{Leastsquare}. Note that by using a sufficiently large number of collocation points, the discrete approximation can be made arbitrarily accurate. This observation will be further confirmed by the numerical experiments presented in Section~\ref{section6}. 
We now return to the construction of the discrete least-squares system. On the chosen collocation points, the optimal solution $\gamma^*$ is approximated by solving
\begin{equation}\label{noregularization}
    J^T J \hat{\gamma}^* = J^T \vec{f},
\end{equation}
where $J$ is the gradient matrix, and $\vec{f}$ is the vector field evaluated at those points, i.e.,  
\begin{equation}\label{matrixandvector}
    (J)_{ij} = \frac{\partial \hat{u}^i}{\partial q_j}, \quad (\vec{f})_i = f(t,x_i,\hat{u}^i, D\hat{u}^i),
\end{equation}
and $i \in [N_C]$ is the index of the collocation points, while $j \in [N_P]$ is the index of the parameters.

Generally, the gradient matrix 
$J$ may have arbitrarily small singular values, which can lead to numerical stability issues and even multiple solutions for the LS problem \eqref{noregularization}. Therefore, it is necessary to introduce an additional regularization term to construct a well-posed problem, i.e., 
\begin{equation}\label{wLeastsquare}
    \mathcal{J}(\gamma) = \left\|\frac{\partial \hat{u}}{\partial q}\gamma - f(t,\cdot,\hat{u},D\hat{u}) \right\|_{L^2}^2 + \lambda^2 \|\gamma\|_{\mathcal{Q}}^2,
\end{equation}
where $\lambda$ is the regularization parameter and $\|\cdot\|_{\mathcal{Q}}$ denotes the norm on $\mathcal{Q}$, which is chosen as the standard $\ell^2$ inner product on $\mathbb{R}^{N_P}$. Then, 
the corresponding discretized version of the first-order optimality condition is 
\begin{equation}\label{normalequation}
    (J^TJ + \lambda^2 I) \gamma^* = J^T \vec{f},
\end{equation}
where $I$ denotes the identity matrix. Both direct inversion methods and iterative methods can be applied to solve \eqref{normalequation}. From a geometric perspective, this procedure can be interpreted as projecting the vector field onto the tangent space of the parametric manifold with respect to a mixed inner product induced by $L^2$ norm and regularization. Beyond its numerical stabilization effect, the regularization also plays a fundamental role in the theoretical analysis. This aspect will be discussed in detail in Section~\ref{section5}. Having outlined the formulation and interpretation of the sequential-in-time training framework, we will complete the description of the algorithm by specifying the treatment of the initial condition, boundary conditions,   and the time integration strategy.

The initial condition is satisfied by solving an optimization problem or directly by interpolation. In this paper, we adopt the deep neural network as the nonlinear parametrization. The initial parameter $q_{0}$ is obtained by optimizing the following loss function, i.e.,
\begin{equation}\label{Loss}
    q_{0} =  \arg\min_{q\in \mathcal{Q}}\sum_{i=1}^{N_u}\left|\hat{u}(x_i) - u_0(x_i)\right|^2,
\end{equation}
using some optimizers such as  Adam \cite{kingma2014adam} and LBFGS \cite{liu1989limited}. 

The treatment of boundary conditions follows standard practices in the neural network–based approach, and we refer the reader to existing literature for detailed implementations. 
For periodic boundary conditions, periodicity is enforced by mapping the input coordinates to trigonometric features, such as $\sin(x)$ and $\cos(x)$, which guarantees periodicity by construction; see, e.g., \cite{yazdani2020systems}. 
For Dirichlet boundary conditions, we adopt commonly used approaches based on multiplying the network output by a distance function to the boundary or incorporating suitable interpolation polynomials, so that the prescribed boundary values are satisfied exactly, see \cite{berg2018unified, luo2024two}. 
Subsequently,  different time integration schemes can be employed to perform the time marching. In this paper, we mainly adopt the explicit Runge-Kutta method, e.g.,  RK3 and RK45 schemes.

For the two fractional Burgers equations considered in this paper, the above procedure is applied by specifying the operator $f$ and assembling the vector $\vec f$ in~\eqref{matrixandvector} accordingly. 
Concretely, at each Runge--Kutta stage and each collocation point $x_i$, we evaluate
$f(t,x_i,\hat u^i,D\hat u^i)$ by combining (i) the integer-order derivatives (e.g., $u_x,u_{xx}$), obtained from the parametric surrogate $\hat u=\Phi(\cdot;q)$ via automatic differentiation, and (ii) the fractional/nonlocal operators, evaluated numerically using the finite difference schemes introduced earlier in this paper.

In detail, for the fractional Laplacian terms appearing in the fractional Burgers equation, we use the shifted Grünwald--Letnikov approximation described in Section~\ref{section_2.2} (cf.~\eqref{GL1}--\eqref{GL2}). 
For the nonlocal nonlinearity in \eqref{FBENN3}, which involves the composition ${}^C_a D^{\beta}\bigl({}^C_a D^{1-\beta}u\bigr)^2$, we compute the inner Caputo operator ${}^C_a D^{1-\beta}u$ and then apply the outer Caputo operator ${}^C_a D^{\beta}$ to the squared result, using the same Caputo discretization scheme introduced earlier (cf.~\eqref{L1}). 
In this way, the nonlocal contributions are incorporated into $\vec f$ while the overall sequential-in-time least-squares framework~\eqref{normalequation}--\eqref{Loss} remains unchanged.

\subsection{Algorithm}
In this part, based on the nonlinear parametric formulation and different discretization schemes we have introduced, we present the detailed steps of our algorithm for the fractional Burgers equations; see \ref{alg:nonlinear_parametric}.

\begin{algorithm}[H]
  \caption{Nonlinear‐Parametric Solver for Fractional Burgers Equations}
  \label{alg:nonlinear_parametric}
  \KwIn{Initial condition \(u_{0}(x)\), collocation points \(\{x_{i}\}_{i=1}^{N_u}\), initial time step \(\Delta t\), final time \(T\), regularization parameter \(\lambda\).}
  \KwOut{Approximate solution \(\hat{u}(x,t)\) via parameter set \(q(t)\).}

  Initialize the parametrization $q(0) \in \mathcal{Q}$ by interpolation or minimizing (\ref{Loss}). \;

  \While{\(t < T\)}{
    Compute spatial fractional derivatives by (\ref{L1}), (\ref{GL1}), (\ref{GL2}) and the integer-order derivatives by automatic differentiation. \;
    
    Construct the gradient matrix $J$ and vector field $\vec{f}$ by (\ref{matrixandvector}). \;
    
    Solve the regularized least-squares problem (\ref{normalequation}).\;

    Perform time marching by Runge-Kutta method, and get the new time step $\Delta t$. \;

    $t \leftarrow t + \Delta t$.\;
  }
  \Return \(\hat{u}(x,t)\) represented by \(\Phi\bigl(x;\,q(t)\bigr)\). \;
\end{algorithm}

\section{Theoretical analysis}\label{section5}
In this section, we analyze the stability of our nonlinear parametric solver and derive an {\it a posteriori}  error estimate by tracking the defect of each least-squares subproblem. Finally, we examine how the order of the finite-difference approximation for fractional derivatives affects the results. We will first introduce some basic notations that will be used in our analysis.

Since the vector field \(f\) contains fractional derivatives that we approximate by a finite-difference scheme, we introduce the discrete vector field \(f^{h}\), where \(h\) denotes the grid size. The corresponding truncation error is
\begin{equation}\label{finite difference error}
R^{h}(t,x)= f(t,x,\hat u,D\hat u)-f^{h}(t,x,\hat u,D\hat u).
\end{equation}
Our goal is to quantify the cumulative effect of this discretization error and the approximation error induced by restricting the dynamics to a parametric manifold. To this end, let
\[
\mathcal{M}=\{\Phi(\cdot;q)\mid q\in\mathcal{Q}\}
\]
be the manifold induced by the parametrization, and let \(T_q\mathcal{M}\) denote its tangent space at \(q\).
At each time \(t\), we approximate the vector field \(f^{h}(t,\cdot,\hat u,D\hat u)\) by an element of \(T_{q(t)}\mathcal{M}\) in a regularized \(L^{2}\) sense.
Specifically, we equip pairs \((u,q)\in L^{2}\times\mathcal{Q}\) with the mixed inner product
\begin{equation}\label{mixed inner product}
((u_1,q_1),(u_2,q_2)) := (u_1,u_2)_{L^{2}}+\lambda^{2}(q_1,q_2)_{\mathcal{Q}},
\end{equation}
where \(\lambda>0\) is a regularization parameter.
The induced projection \(P_{q}^{\lambda}\) onto \(T_q\mathcal{M}\) is realized by the following Tikhonov-regularized least-squares problem:
\begin{equation}\label{least-squares problem 5}
\dot q(t)=\arg\min_{\gamma\in\mathcal{Q}}
\left\|\frac{\partial \hat u}{\partial q}(t)\,\gamma - f^{h}(t,\cdot,\hat u,D\hat u)\right\|_{L^{2}}^{2}
+\lambda^{2}\|\gamma\|_{\mathcal{Q}}^{2}.
\end{equation}
With this choice of \(\dot q(t)\), the projected vector field is
\begin{equation}\label{projector}
P_{q(t)}^{\lambda} f^{h}(t,x,\hat u,D\hat u)
:= \frac{\partial \hat u}{\partial q}(t,x)\,\dot q(t)\in T_{q(t)}\mathcal{M}.
\end{equation}
The mismatch between the discrete vector field and its tangent space approximation is measured by the  defect function
\begin{equation}\label{defect function}
d^{\lambda,h}(t,x)
:= f^{h}(t,x,\hat u,D\hat u)-P_{q(t)}^{\lambda} f^{h}(t,x,\hat u,D\hat u),
\end{equation}
and we define the corresponding least-squares residual as
\begin{equation}\label{delta_function}
\delta^{\lambda,h}(t)
:= \left(\left\|\frac{\partial \hat u}{\partial q}(t)\,\dot q(t)- f^{h}(t,\cdot,\hat u,D\hat u)\right\|_{L^{2}}^{2}
+\lambda^{2}\|\dot q(t)\|_{\mathcal{Q}}^{2}\right)^{\frac{1}{2}}.
\end{equation}
Since \(\delta^{\lambda,h}(t)\) contains the squared \(L^{2}\)-norm of the defect, we immediately obtain
\begin{equation}\label{residual and defect}
\|d^{\lambda,h}(t,\cdot)\|_{L^{2}}\le \delta^{\lambda,h}(t).
\end{equation}
Therefore, tracking \(\delta^{\lambda,h}(t)\) provides a computable {\it a posteriori}  indicator of how well the parametric dynamics follows the discretized vector field \(f^{h}\).

Finally, by construction, the numerical solution evolves along the manifold according to the projected dynamics
\begin{equation}\label{evolution equation}
\hat u_t(t,x)
= P_{q(t)}^{\lambda} f^{h}(t,x,\hat u,D\hat u)
= f^{h}(t,x,\hat u,D\hat u)-d^{\lambda,h}(t,x).
\end{equation}
The dynamics \eqref{evolution equation} is fully determined once the projection operator \(P_q^\lambda\) is specified.
In practice, however, the Jacobian \(\partial \hat u/\partial q\) may be poorly conditioned, making the unregularized \(L^2\) projection sensitive to discretization and modeling errors.
To address this issue, we introduce the regularized projection \(P_q^\lambda\) and analyze its stabilizing effect.
We formalize this point in the next subsection by comparing the \(L^2\) projection  with the regularized \(L^2\) projection.

\subsection{Regularization}
In this section, we compute the operator norm of the regularized projection and
clarify the role of regularization in the stability of the proposed algorithm.

At each time step, we solve the least-squares problem
\eqref{least-squares problem 5}, whose solution can be written as
\begin{equation}
    \label{q solution}
    \dot{q}
    =
    \left(
        \left( \frac{\partial \hat{u}}{\partial q} \right)^*
        \left( \frac{\partial \hat{u}}{\partial q} \right)
        + \lambda^2 I
    \right)^{-1}
    \left( \frac{\partial \hat{u}}{\partial q} \right)^*
    f^h,
\end{equation}
where $\frac{\partial \hat{u}}{\partial q} : \mathcal Q \to L^2$ denotes the
Jacobian of the parametrization, and
$\left(\frac{\partial \hat{u}}{\partial q}\right)^*$ is its $L^2$-adjoint.
The corresponding projected vector field is given by
\begin{equation}
    P^{\lambda}_{q} f^h
    =
    \left( \frac{\partial \hat{u}}{\partial q} \right)
    \left(
        \left( \frac{\partial \hat{u}}{\partial q} \right)^*
        \left( \frac{\partial \hat{u}}{\partial q} \right)
        + \lambda^2 I
    \right)^{-1}
    \left( \frac{\partial \hat{u}}{\partial q} \right)^*
    f^h .
\end{equation}

Let $\{\sigma_i\}_{i=1}^{N_P}$ denote the singular values of
$\frac{\partial \hat{u}}{\partial q}$, equivalently, the square roots of the
nonzero eigenvalues of
$\left( \frac{\partial \hat{u}}{\partial q} \right)^*
 \left( \frac{\partial \hat{u}}{\partial q} \right)$.
We further denote by $\sigma_{\min}$ and $\sigma_{\max}$ the smallest and largest
singular values, respectively.
Then the spectrum of the regularized projection satisfies
\begin{equation}
    \operatorname{spec}(P^\lambda_q)
    =
    \left\{
        \frac{\sigma_i^2}{\sigma_i^2 + \lambda^2},
        \quad i = 1, \dots, N_P
    \right\},
\end{equation}
and its induced operator norm on $L^2$ admits the bound
\begin{equation}
    \label{norm of projection}
    \|P^\lambda_q\|
    \le
    \frac{\sigma_{\max}^2}{\sigma_{\max}^2 + \lambda^2}
    < 1 .
\end{equation}
Therefore, the regularized projection operator $P^\lambda_q$ is a contraction on
$L^2$, a property that will be used in the stability estimate, i.e., Theorem
\ref{stability estimate}.

Although the regularized projection $P_q^\lambda$ has favorable contraction
property in $L^2$, this does not preclude strong sensitivity in the parameter evolution, which we now quantify.
As discussed in Section \ref{wLeastsquare}, the least-squares subproblem may be severely ill-conditioned, corresponding to $\sigma_{\min} \approx 0$.
The induced operator norm in \eqref{q solution} from $L^2$ to the parameter space $\mathcal Q$ satisfies
\begin{equation}
\left\|
\left(
    \left( \frac{\partial \hat{u}}{\partial q} \right)^*
    \left( \frac{\partial \hat{u}}{\partial q} \right)
    + \lambda^2 I
\right)^{-1}
\left( \frac{\partial \hat{u}}{\partial q} \right)^*
\right\|
=
\max_{1 \le i \le N_P}
\frac{\sigma_i}{\sigma_i^2 + \lambda^2}
\le
\frac{1}{2\lambda}.
\end{equation}
In contrast, in the unregularized case $\lambda = 0$, the corresponding operator
satisfies
\begin{equation}
\left\|
\left(
    \left( \frac{\partial \hat{u}}{\partial q} \right)^*
    \left( \frac{\partial \hat{u}}{\partial q} \right)
\right)^{-1}
\left( \frac{\partial \hat{u}}{\partial q} \right)^*
\right\|
=
\frac{1}{\sigma_{\min}},
\end{equation}
which becomes unbounded as $\sigma_{\min} \to 0$.

As a consequence, any perturbation in the vector field—most notably the
discretization error arising from finite-difference approximations of fractional
derivatives—leads to
\begin{equation}
\|\delta \dot q\|
\lesssim
\frac{1}{\lambda}\,\|\delta f^h\|,
\end{equation}
whereas without regularization the amplification factor scales like
$1/\sigma_{\min}$.
Therefore, ill-conditioning does not affect the $L^2$-operator norm of the
regularized projection $P_q^\lambda$, but instead enters through the mapping
$f^h \mapsto \dot q$, making the parameter evolution highly sensitive to small
perturbations in the vector field.
The regularization stabilizes the algorithm by uniformly controlling this
sensitivity and preventing discretization errors from being strongly amplified in
the parameter dynamics.

\subsection{Stability}
The \( L^2 \) energy of the solution of FBEFL will monotonously decay with the homogeneous Dirichlet boundary condition due to the dissipation effect of \( (-\Delta)^{\frac{\alpha}{2}} \). Our numerical solution will nearly obey the dissipation law. However, since the projection is based on the regularized \( L^2 \) projection instead of \( L^2 \) projection, this will introduce some energy drift effect depending on \( h \) and \( \lambda \).

\begin{theorem}[$L_{2}$ stability estimate]
\label{stability estimate}
    Let $\hat{u}(x,t) = \Phi(x;q(t))$ be the parametric numerical solution of FBEFL \eqref{FBEFL2}. If $\hat{u}(x,t) \in H^{\alpha/2}_0$, then 
    \begin{equation*}
        \|\hat{u}(\cdot,t)\|_{L^2} \leq e^{-\lambda^*_\alpha \varepsilon t}\|\hat{u}(\cdot,0)\|_{L^2} + \int_0^t e^{-\lambda^*_\alpha \varepsilon (t-s)} \left( \delta^{\lambda,h}(s) +\|R^h(s,\cdot)\|_{L^2} \right) \,\mathrm{d}s, 
        \end{equation*}
    where $\lambda^*_\alpha > 0$ is the smallest eigenvalue of $(-\Delta)^{\frac{\alpha}{2}}$ with homogeneous Dirichlet boundary conditions.
    \end{theorem}
    \begin{proof}
    In FBEFL, let 
    \begin{equation}
        f(\hat{u},D\hat{u}) = -\varepsilon(-\Delta)^{\frac{\alpha}{2}}\hat{u} - \hat{u}\hat{u}_x
    \end{equation}
    be the corresponding vector field.
    By the evolution equation \eqref{evolution equation}, we have
    \begin{align}
\frac{\mathrm{d}}{\mathrm{d}t}\frac12\|\hat{u}(\cdot,t)\|_{L^2}^2
&= \langle \hat{u}(\cdot,t), \partial_t \hat{u}(\cdot,t) \rangle \nonumber\\
&= \langle \hat{u}, P^\lambda_{q(t)} f^h \rangle \nonumber\\
&= \langle \hat{u}, P^\lambda_{q(t)}(f - R^h) \rangle \nonumber\\
&= \big\langle \hat{u}, -\varepsilon(-\Delta)^{\frac{\alpha}{2}}\hat{u}
        - \hat{u}\hat{u}_x - d^{\lambda,h} \big\rangle
   - \langle \hat{u}, P^\lambda_{q(t)} R^h \rangle .
\label{eq:energy-identity}
\end{align}
Using \eqref{norm of projection}, and noting that
\(\langle \hat{u},\hat{u}\hat{u}_x\rangle=0\) under homogeneous Dirichlet boundary conditions, we obtain
\begin{align}
\frac{\mathrm{d}}{\mathrm{d}t}\frac12\|\hat{u}(\cdot,t)\|_{L^2}^2
&\le -\lambda_\alpha^\ast \varepsilon \,\|\hat{u}(\cdot,t)\|_{L^2}^2
   + \|\hat{u}(\cdot,t)\|_{L^2}\,\|d^{\lambda,h}(t,\cdot)\|_{L^2} \nonumber\\
&\quad + \|\hat{u}(\cdot,t)\|_{L^2}\,
\|R^h(t,\cdot)\|_{L^2},
\label{eq:energy-ineq-1}
\end{align}
where we also use the spectral property of $(-\Delta)^{\frac{\alpha}{2}}$ \cite{chen2005two}, i.e., 
        \begin{equation} \label{spectral property}
            \langle \hat{u}, (-\Delta)^{\frac{\alpha}{2}}\hat{u} \rangle \geq \lambda^*_\alpha \|\hat{u}\|_{L^2}^2. 
        \end{equation}
Finally, by \eqref{residual and defect}, we have
\begin{align}
\frac{\mathrm{d}}{\mathrm{d}t}\frac12\|\hat{u}(\cdot,t)\|_{L^2}^2
&\le -\lambda_\alpha^\ast \varepsilon \, \|\hat{u}(\cdot,t)\|_{L^2}^2
   + \|\hat{u}(\cdot,t)\|_{L^2}\,\delta^{\lambda,h}(t) \nonumber\\
&\quad + \|\hat{u}(\cdot,t)\|_{L^2}\,
\|R^h(t,\cdot)\|_{L^2},
\label{eq:energy-ineq-2}
\end{align}
        Using the Gr\"onwall inequality, we finish the proof,
        \begin{equation}
        \|\hat{u}(\cdot,t)\|_{L^2} \leq e^{-\lambda^*_\alpha \varepsilon t}\|\hat{u}(\cdot,0)\|_{L^2} + \int_0^t e^{-\lambda^*_\alpha \varepsilon (t-s)} \left( \delta^{\lambda,h}(s) +\|R^h(s,\cdot)\|_{L^2} \right) \,\mathrm{d}s.
        \end{equation}
    \end{proof}

    The first term on the right-hand side arises directly from the dissipative effect of \( (-\Delta)^{\alpha/2} \), whereas the second term accounts for the energy drift, which is driven by
    \begin{itemize}
        \item the computable projection defect $\delta^{\lambda,h}$,
        \item the spatial finite difference truncation error $\|R^h\|_{L^2}$.
    \end{itemize}
So it is instructive to consider the limiting regime \(\lambda\to 0\) and \(h\to 0\).
In this case, the spatial perturbation terms vanish and the regularized projection reduces to an exact \(L^2\) projection.
Consequently, one recovers a strict energy dissipation law, which serves as a baseline consistency check for our scheme and suggests that the drift observed for \(\lambda>0\) should be small when \(\lambda\) is sufficiently small and the grid is sufficiently refined.
However, to recover an exact energy dissipation identity, it is not sufficient to rely solely on the orthogonality of the projection.
One must additionally ensure that the numerical solution itself lies in the
range of the projection operator.
This requires a structural assumption on the parametrization $\Phi$, which we state next.
\begin{assumption}[Linear-in-parameters structure]
\label{assumption}
     There exists a subset of parameters $a \in \mathbb{R}^m$ (e.g. last-layer weights) with $m < N_p$ and remaining parameters $\theta$ such that
\begin{equation}
    \Phi(x;a,\theta) = \sum_{j=1}^m a_j \psi_j(x;\theta)
\end{equation}
with $\psi(·; \theta) \in L^2$ for each $\theta$. Then, for $q=(a,\theta)$ and all $t$, the numerical solution
$\hat{u}(\cdot,t)=\Phi(\cdot;q(t))$ belongs to the tangent space, i.e., \begin{equation}
\label{belong to tangent space}
    \hat{u}(·, t) \in T_q\mathcal{M}.
\end{equation}
Hence $P_{q(t)}\hat{u}(t) = \hat{u}(t)$ for the orthogonal (unregularized) projection $P_q := P^{\lambda = 0}_q$.
\end{assumption}

\begin{remark}
In fact, Assumption \ref{assumption} holds for a wide range of parametrizations \cite{zhang2024sequential}, including feed-forward deep neural networks, as long as the last layer is linear. Thus, it is reasonable to assume it in our theorem.
\end{remark}

This assumption ensures that, in the unregularized and continuum limit, i.e. $\lambda = 0$ and $h \to 0$, the projection operator acts trivially on the numerical solution.
As a result, the projected dynamics coincides with the full vector field when tested against $\hat{u}$.
This property is essential for recovering an exact energy dissipation identity, as shown in the following theorem.

    \begin{theorem}[Exact $L^2$ dissipation in the unregularized limit]
    Assume that the
parametrization $\Phi(\cdot;q)$ satisfies Assumption \ref{assumption}.
        Let $\hat{u}(x,t) = \Phi(x;q(t))$ be the parametric numerical solution of FBEFL (\ref{FBEFL2}). Consider the limiting regime $\lambda = 0$ and $h \to 0$
        and the projected dynamics
        \begin{equation}
        \label{unregularized dynamic}
            \hat{u}_t = P_{q(t)}f(t,x,\hat{u},D\hat{u}), \quad \text{with} \; f(t,x,\hat{u},D\hat{u}) = -\varepsilon(-\Delta)^{\frac{\alpha}{2}}\hat{u} - \hat{u}\hat{u}_x,
        \end{equation}
where $P_q$ is the $L^2$ projection onto $T_q \mathcal{M}$. If $\hat{u}(x,t) \in H_0^{\alpha/2}$, then 
    \begin{equation*}
        \|\hat{u}(\cdot,t)\|_{L^2} \leq e^{-\lambda^*_\alpha \varepsilon t}\|\hat{u}(\cdot,0)\|_{L^2}, 
    \end{equation*}
    \end{theorem}
    \begin{proof}
        The proof follows a standard energy argument; see, e.g.,
        \cite{zhang2024sequential}. Since $P_q$ is the $L^2$ projection 
        it is self-adjoint and satisfies
        \begin{equation}
        \label{self-adjoint}
        \langle P_q u, v \rangle_{L^2} = \langle u, P_q v \rangle_{L^2}.
        \end{equation}
        By the projected dynamics \eqref{unregularized dynamic}, we have
        \begin{align}
        \frac{\mathrm{d}}{\mathrm{d}t}\frac{1}{2}\|\hat{u}(\cdot,t)\|_{L^2}^2 & = \langle \hat{u}(\cdot,t), \partial_t \hat{u}(\cdot,t) \rangle \nonumber\\
            & = \langle\hat{u}, P_{q(t)}[-\varepsilon(-\Delta)^{\frac{\alpha}{2}}\hat{u} - \hat{u}\hat{u}_x] \rangle
        \end{align}
        By \eqref{belong to tangent space} and the self-adjointness \eqref{self-adjoint}, we obtain
        \begin{align}
            \frac{\mathrm{d}}{\mathrm{d}t}\frac{1}{2}\|\hat{u}(\cdot,t)\|_{L^2}^2 & = \langle P_{q(t)}\hat{u}, -\varepsilon(-\Delta)^{\frac{\alpha}{2}}\hat{u} - \hat{u}\hat{u}_x \rangle \nonumber \\
            & = \langle \hat{u}, -\varepsilon(-\Delta)^{\frac{\alpha}{2}}\hat{u} - \hat{u}\hat{u}_x \rangle,
        \end{align}
         Using integration by parts and \eqref{spectral property}, we can obtain the result. 
    \end{proof}

Before proceeding to the error analysis, we summarize the roles of the two key quantities in our framework: the least-squares residual \(\delta^{\lambda,h}\) and the finite-difference truncation error \(R^h\).

In the stability estimate above, these quantities act as perturbation sources that quantify the deviation of the numerical flow from the exact dissipative dynamics.
Importantly, \(\delta^{\lambda,h}\) is fully computable from the iterates at each step, while \(R^h\) admits a known convergence behavior as \(h\) is refined.
Therefore, they provide practical {\it a posteriori} indicators for diagnosing energy drift and potential loss of stability.

In the next subsection, we show that the same pair \((\delta^{\lambda,h}, R^h)\) also enters the evolution of the numerical error and leads to an {\it a posteriori} \(L^2\)-error bound.
This yields a unified perspective: these quantities not only explain the energy drift in the stability analysis, but also control the growth of the numerical error.

\subsection{Error estimate}
In our algorithm, the vector field is approximated in two steps. First, we use a finite difference scheme to approximate the fractional derivative, denoted by \( f^h \), which approximates the original \( f \). Then, we project \( f^h \) onto the tangent space \( T_q \mathcal{M} \). In our error estimate, we account for the errors introduced in both of these steps. Based on this approach, the following theorem provides an error estimate for the numerical solution.

Let the exact solution of FBEFL \eqref{FBEFL2} satisfy 
\begin{equation}
    \label{exact evolution}
    u_t = f(t, x, u, Du),    
\end{equation} and the corresponding numerical solution satisfy $\hat{u}_t = P^\lambda_{q(t)} f_h$ with $f_h = f + R^h$ and $\|d^{\lambda,h}(t)\|_{L^2} \leq \delta^{\lambda,h}(t)$ as in \eqref{residual and defect} and \eqref{evolution equation}.
\begin{theorem}[Error estimate] \label{error estimate}
     Assume that $\exists M >0$, such that $\|\hat{u}(\cdot,t)\|_{W^{1,\infty}} \leq M$ and $\|u(\cdot,t)\|_{W^{1,\infty}} \leq M$. If $\hat{u}(x,t) \in H_0^{\alpha/2}$, then it holds
    \begin{align*}
        \|u(\cdot,t) - \hat{u}(\cdot,t)\|_{L^2} & \leq e^{(\frac{3}{2}M - \lambda_\alpha^* \varepsilon )t}\|u(\cdot,0) - \hat{u}(\cdot, 0)\|_{L^2} + \int_0^t e^{(\frac{3}{2}M - \lambda^*_\alpha \varepsilon)(t-s)}\delta^{\lambda,h}(s) \,\mathrm{d}s \nonumber\\
        & \quad + \int_0^t e^{(\frac{3}{2}M - \lambda^*_\alpha \varepsilon)(t-s)} \|R^h(s,\cdot)\|_{L^2} \,\mathrm{d}s.
    \end{align*}
\end{theorem}
\begin{proof}
By \eqref{evolution equation} and \eqref{exact evolution}, we obtain
    \begin{align}
        \frac{1}{2}\frac{\mathrm{d}}{\mathrm{d}t}\|u - \hat{u}\|_{L^2}^2 & = \langle u - \hat{u}, \partial_t u - \partial_t \hat{u} \rangle \nonumber\\
        & = \langle u - \hat{u}, f(t,\cdot,u,Du) - P^\lambda_{q(t)} f^h(t,\cdot,\hat{u},D\hat{u}^h) \rangle \nonumber\\
        & = \langle u - \hat{u},f(t,\cdot,u,Du) - f(t,\cdot,\hat{u},D\hat{u}) \rangle \nonumber\\
        & \quad + \langle u - \hat{u}, f(t,\cdot,\hat{u},D\hat{u}) - f^h(t,\cdot,\hat{u},D\hat{u}) \rangle \nonumber\\
        & \quad + \langle u - \hat{u},f^h(t,\cdot,\hat{u},D\hat{u}) - P^\lambda_{q(t)} f^h(t,\cdot,\hat{u},D\hat{u}^h) \rangle,
    \end{align}
    where we use the add-and-subtract technique in the last equality. We will estimate these three terms respectively. The first term is estimated by standard integration by parts,
    \begin{align}
        \langle u - \hat{u},f(t,\cdot,u,Du) - f(t,\cdot,\hat{u},D\hat{u}) \rangle & = \langle\hat{u} - u , \hat{u}\hat{u}_x + \varepsilon (-\Delta)^{\frac{\alpha}{2}}\hat{u} - uu_x - \varepsilon (-\Delta)^{\frac{\alpha}{2}}u \rangle \nonumber\\
        & \leq - \lambda^*_\alpha \varepsilon \|u - \hat{u}\|_{L^2}^2 + |\langle\hat{u} - u ,\hat{u}\hat{u}_x - uu_x \rangle|.
        \label{com 1}
    \end{align}
    By assumption, $\|\hat{u}(\cdot,t)\|_{W^{1,\infty}}$ and $\|{u}(\cdot,t)\|_{W^{1,\infty}}$ are uniformly bounded by $M$,
    \begin{align}
       |\langle\hat{u} - u, \hat{u}\hat{u}_x - uu_x \rangle| & \leq |\langle\hat{u} - u, \hat{u}\hat{u}_x - \hat{u}u_x \rangle| +  |\langle\hat{u} - u, \hat{u}u_x - uu_x \rangle| \nonumber \\
       & = |\langle \frac{1}{2}(\hat{u} - u)^2, \hat{u}_x \rangle| + |\langle (\hat{u} - u)^2, u_x \rangle| \nonumber \\
       & \leq (\frac{1}{2}\|\hat{u}_x\|_{L^\infty} + \|u_x\|_{L^\infty}) \|\hat{u} - u\|_{L^2}^2 \nonumber \\
       & \leq \frac{3}{2}M\|\hat{u} - u\|_{L^2}^2. \label{com 2}
    \end{align}
    By definition \eqref{defect function}, \eqref{finite difference error}, inequality \eqref{residual and defect} and Cauchy-Schwartz inequality, the second and third term are estimated by
    \begin{equation}
    \label{com 3}
        |\langle u - \hat{u}, f(t,\cdot,\hat{u},D\hat{u}) - f^h(t,\cdot,\hat{u},D\hat{u}) \rangle| \leq \|u - \hat{u}\|_{L^2} \|R^h\|_{L^2},
    \end{equation}
    and
    \begin{equation}
    \label{com 4}
        |\langle u - \hat{u},f^h(t,\cdot,\hat{u},D\hat{u}) - P^\lambda_{q(t)} f^h(t,\cdot,\hat{u},D\hat{u}^h) \rangle| \leq \|u - \hat{u}\|_{L^2} \delta^{\lambda,h}(t).
    \end{equation}
    Combining \eqref{com 1}, \eqref{com 2}, \eqref{com 3}, \eqref{com 4}, and applying Gr\"onwall inequality, we derive the final error estimate,
    \begin{align}
        \|u(\cdot,t) - \hat{u}(\cdot,t)\|_{L^2} & \leq e^{(\frac{3}{2}M - \lambda_\alpha^* \varepsilon )t}\|u(\cdot,0) - \hat{u}(\cdot, 0)\|_{L^2} + \int_0^t e^{(\frac{3}{2}M - \lambda^*_\alpha \varepsilon)(t-s)}\delta^{\lambda,h}(s) \,\mathrm{d}s \nonumber\\
        & \quad + \int_0^t e^{(\frac{3}{2}M - \lambda^*_\alpha \varepsilon)(t-s)} \|R^h(s,\cdot)\|_{L^2} \,\mathrm{d}s.
    \end{align}
\end{proof}
The above theorem implies that the error can be decomposed into three components:
\begin{itemize}
    \item Initial condition fitting error \( \|u(\cdot,0) - \hat{u}(\cdot, 0)\|_{L^2} \), which can be reduced by increasing the number of parameters or the number of optimization iterations.
    \item Least-squares residual \( \delta^{\lambda,h}(s) \), which depends on how well the tangent space \( T_q{\mathcal{M}} \) approximates the right-hand side. If stability is not a concern, reducing the regularization term can also help minimize this error.
    \item Truncation error of finite difference approximation \( R^h(t,x) \), which exhibits a clear convergence behavior with $h \to 0$ and can be reduced by using higher-order schemes or employing a finer grid.
\end{itemize}
Our experiments in the next section indicate that once the initial condition fit is sufficiently accurate and the finite difference error dominates, both increasing the order of the finite difference scheme and refining the grid help reduce the overall error. However, as soon as the least-squares residual dominates, the error levels off, and the model's fitting capacity becomes saturated.

\section{Numerical experiments}\label{section6}
In this section, both of the two types of fractional Burgers equations are solved by our method to demonstrate its capability and accuracy. To highlight the advantages of our method over conventional finite difference method and spectral method, we also present comparisons of results using different methods. In addition, to validate the preceding theoretical analysis, we further investigate the relationship between the error of numerical solution and the truncation error of the finite difference approximation, to confirm the correctness of the theoretical results.

In all the following tests, we adopt the multilayer perceptron as the nonlinear approximation,  and we use $\tanh$ as activation function. The initial parameters are initialized using the Xavier initialization scheme. The initial condition is satisfied by minimizing \eqref{Loss}, with the L-BFGS optimizer. The specific boundary conditions are enforced into the structure through the design of the network \cite{du2021evolutional}. The implementation of this approach will be illustrated in the specific examples. 
All computations are performed on a standalone NVIDIA GPU RTX 6000.

To compare the performance between different methods, we will use the relative
$L^2$ error defined as
\begin{equation}
    e = \sqrt{\frac{\sum_{i=1}^{N_t} |u(t_i,x_i) - \hat{u}(t_i,x_i)|^2}{\sum_{i=1}^{N_t}|u(t_i,x_i)|^2 }},
\end{equation}
where $N_t$ is the number of test points in the spatial and temporal domain, and $(t_i,x_i)$ denote the $i$-th test point. 

\subsection{FBEFL} In this part, we solve \eqref{FBEFL2} using STNP.
We present two numerical experiments. In the first experiment, we validate the correctness of STNP by comparing it with a high-resolution shock-capturing reference solver (WENO--Roe) and a baseline central-difference discretization. In the second experiment, we examine the convergence behavior of STNP with respect to the order of the finite difference approximation for the fractional Laplacian.

We first consider \eqref{FBEFL2} on \([-1,1]\) with the initial and boundary conditions
\begin{equation}
u_0(x) = -\sin(\pi x), \qquad u(\pm 1,t)=0.
\end{equation}
To highlight the capability of STNP to suppress spurious oscillations, we choose a small viscosity coefficient \(\varepsilon=0.01\) and fix the fractional order to \(\alpha=1.6\).
For STNP, we use a fully connected neural network with architecture \([1,10,10,1]\), i.e., two hidden layers of width \(10\), resulting in \(141\) trainable parameters, and employ \(300\) uniformly distributed collocation points in space.
As a reference solver, we implement the WENO--Roe scheme following \cite{cockburn1998essentially} using \(201\) degrees of freedom.
Both methods use the third-order strong-stability-preserving Runge--Kutta integrator (SSP-RK3) \cite{shu1988efficient} with time step \(\Delta t=10^{-3}\).
We plot the predicted solutions at \(t=0.5\) and \(t=1.0\) and compare STNP with the central-difference and WENO--Roe schemes in Figure~\ref{FBEFL_comparison}.

\begin{figure}[htbp]
    \centering
    \begin{overpic}[width=0.45\textwidth]{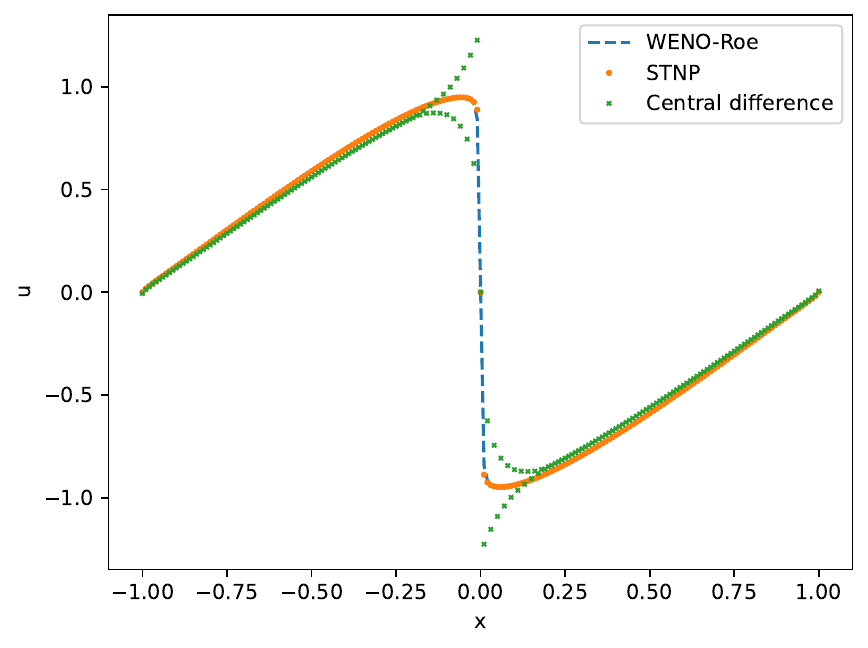}
    \end{overpic}\hfill
    \begin{overpic}[width=0.45\textwidth]{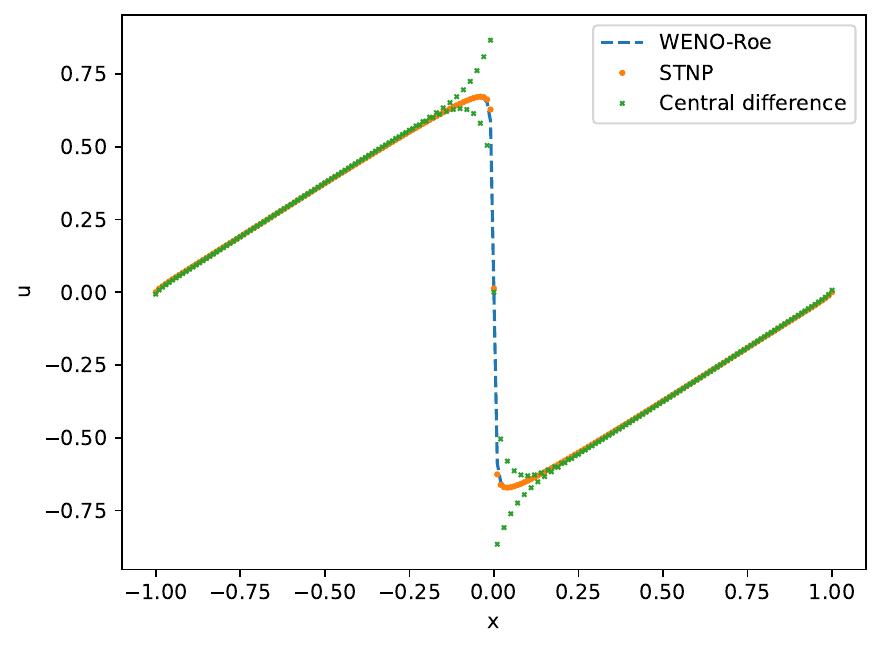}
    \end{overpic}
    \caption{Comparison of numerical solutions computed by STNP, WENO--Roe, and central-difference schemes. Both STNP and WENO--Roe suppress spurious oscillations, whereas the central-difference scheme exhibits pronounced oscillations near the shock.}
    \label{FBEFL_comparison}
\end{figure}

Under these settings, the central-difference baseline develops pronounced numerical oscillations near the shock due to insufficient numerical dissipation.
In contrast, STNP produces non-oscillatory solutions that closely match the WENO--Roe results, while avoiding the nonlinear reconstruction procedures required by WENO-type methods.
This comparison confirms both the accuracy of STNP and its robustness in the low-viscosity regime.

Next, we assess the convergence behavior of STNP as the grid is refined and as the order of the finite-difference approximation for the fractional Laplacian increases.
To isolate truncation errors from other sources, we adopt a manufactured smooth solution,
\[
u(x,t)=e^{-t}x^3(1-x)^3,\qquad x\in[0,1],
\]
and set \(\varepsilon=1.0\) so that the solution remains sufficiently regular.
The corresponding forced problem takes the form
\begin{equation}
u_t + \frac{1}{2}(u^2)_x = -\varepsilon (-\Delta)^{\frac{\alpha}{2}} u + f(x,t),
\end{equation}
where the forcing term \(f(x,t)\) is chosen so that the exact solution is given by the expression above:
\begin{align*}
f(x,t) & = -e^{-t}(x^3(1-x)^3)+e^{-2t}(3-6x)x^5(1-x)^5 \\
& \quad + \frac{\varepsilon}{2\cos(\frac{\alpha\pi}{2})} \Big( \frac{\Gamma(4)}{\Gamma(4-\alpha)}\left( x^{3-\alpha} +(1-x)^{3-\alpha} \right) \\
& \qquad\qquad -\frac{3\Gamma(5)}{\Gamma(5-\alpha)}\left( x^{4-\alpha} +(1-x)^{4-\alpha} \right) \\
& \qquad\qquad +\frac{3\Gamma(6)}{\Gamma(6-\alpha)}\left( x^{5-\alpha} +(1-x)^{5-\alpha} \right) \\
& \qquad\qquad - \frac{\Gamma(7)}{\Gamma(7-\alpha)}\left( x^{6-\alpha} +(1-x)^{6-\alpha} \right)\Big).
\end{align*}
We use RK3 with \(\Delta t=10^{-3}\).
We approximate the fractional Laplacian using different order finite-difference formulas and refine the grid by increasing the number of collocation points
\[
N\in\{5,10,20,40,50,100,200,400,800,1000\}.
\]
To avoid having the network approximation error dominate the results, we keep the network width fixed at \(10\) and increase the depth to \(3\), \(5\), and \(7\). This choice reduces the initial-condition misfit and the residual error so that the observed convergence is primarily driven by the truncation error of the finite-difference approximation. The relative \(L^2\) errors are reported in Figure~\ref{truncation_error}.%

\begin{figure}[htbp]
    \centering
    \begin{overpic}[width=0.31\textwidth]{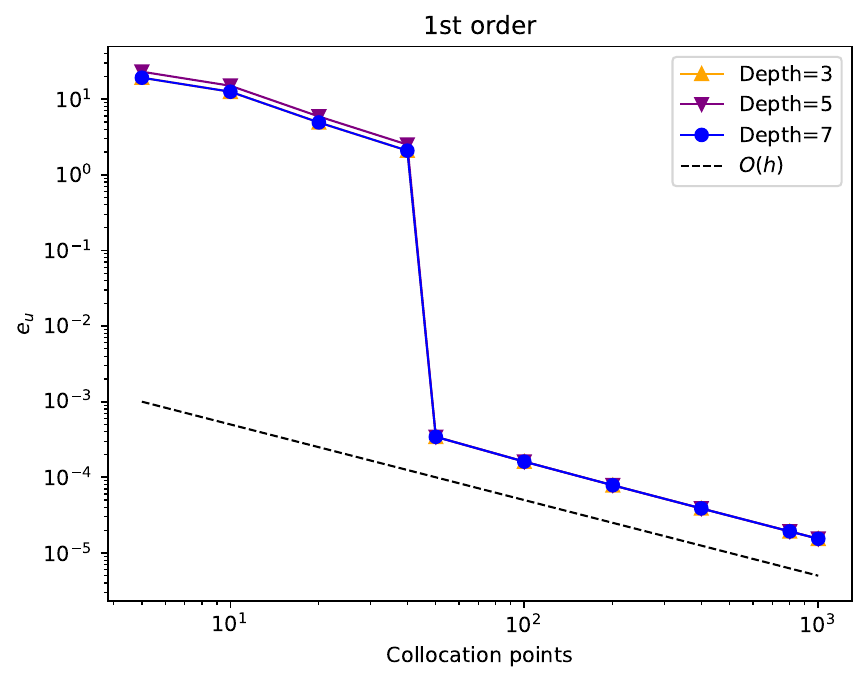}
    \end{overpic}\hfill
    \begin{overpic}[width=0.31\textwidth]{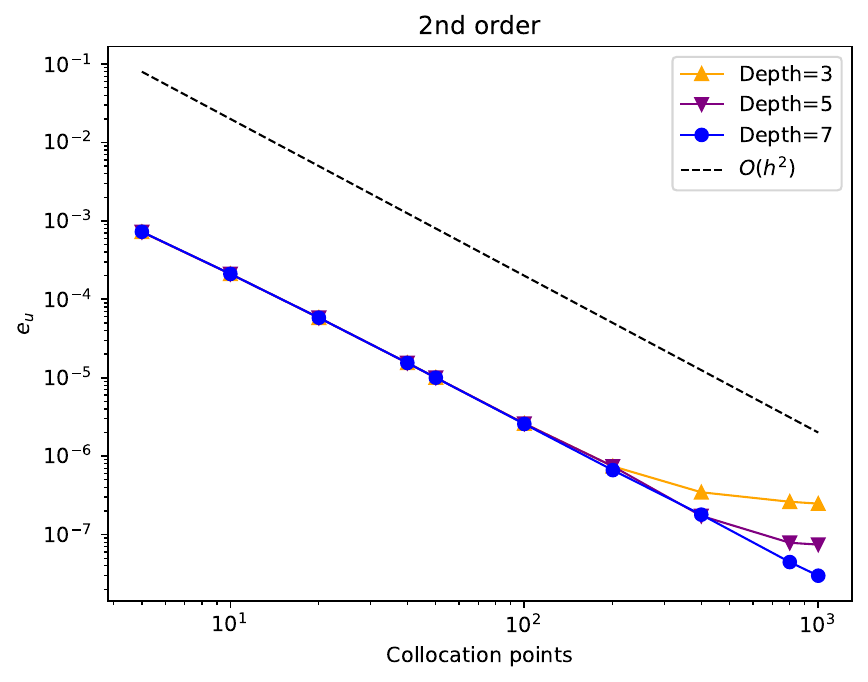}
    \end{overpic}\hfill
    \begin{overpic}[width=0.31\textwidth]{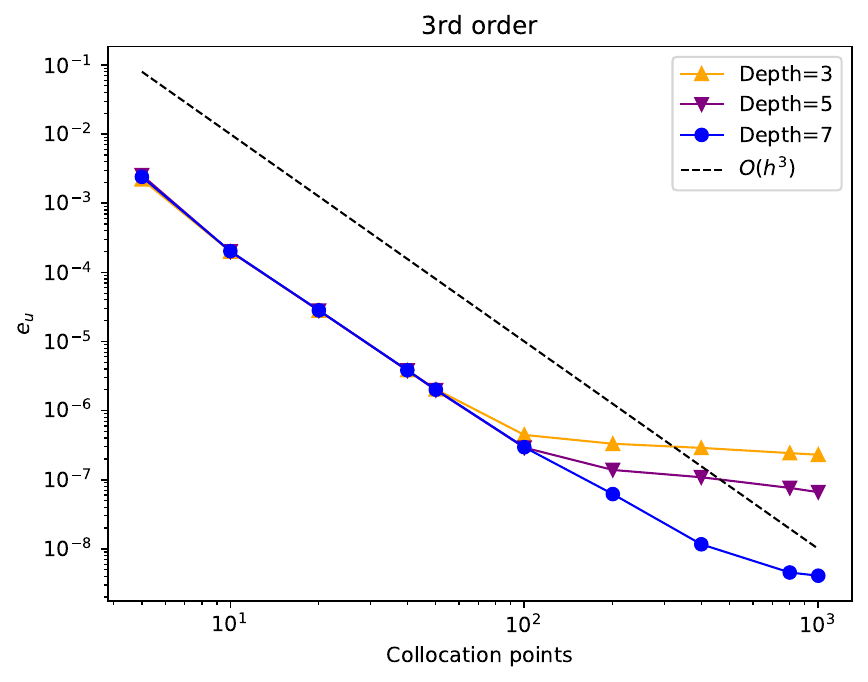}
    \end{overpic}
    \caption{Relative \(L^2\) error versus the number of collocation points \(N\) using first-order (left), second-order (middle), and third-order (right) finite-difference approximations of the fractional Laplacian, for networks of fixed width \(10\) and depths 3, 5, and 7.}
    \label{truncation_error}
\end{figure}

As expected, when \(N\) is small the overall error is dominated by the truncation error, and curves corresponding to different network depths nearly overlap.
As \(N\) increases, the truncation error decreases and the benefit of richer network parametrizations becomes visible: deeper networks reduce the initial-condition misfit and the residual error, leading to smaller final errors and revealing the anticipated convergence trend for each finite-difference order.

\subsection{FBENN} 
In this section, we solve \eqref{FBENN3} using STNP. The initial condition is
\begin{equation}
u_0(x) = -\sin(\pi x),
\end{equation}
and the boundary condition is prescribed by the exact solution obtained from \eqref{exact solution}. The diffusion coefficient is fixed to \(\varepsilon=\frac{1}{150\pi}\).
We perform three sets of tests. We first benchmark STNP against a spectral method with spectral vanishing viscosity (SVV) \cite{mao2017fractional} to verify accuracy and shock-capturing performance. We then investigate the influence of time integration by comparing a fixed-step scheme (RK3) with an adaptive scheme (RK45). Finally, we examine the long-time behavior of STNP and its sensitivity to random initialization.

We first solve \eqref{FBENN3} using STNP and the spectral-SVV method. For STNP, we use a fully connected network with architecture \([1,10,10,1]\). For the spectral method, we employ Jacobi polynomials of degree \(128\). Both solvers use the third-order Runge--Kutta method (RK3) with time step \(\Delta t=10^{-3}\). Figure~\ref{FBENN_comparison} compares the numerical and exact solutions at \(t=0.5\) and \(t=1.0\) for \(\beta=0.6\) and \(\beta=0.8\), respectively. The spectral method requires a relatively large SVV to suppress the Gibbs phenomenon, which inevitably smears discontinuities and reduces accuracy near shocks. In contrast, STNP resolves the shock location more sharply and matches the exact solution more closely, leading to visibly improved accuracy.

\begin{figure}[htbp]
    \centering
    \begin{overpic}[width = 0.45\textwidth]{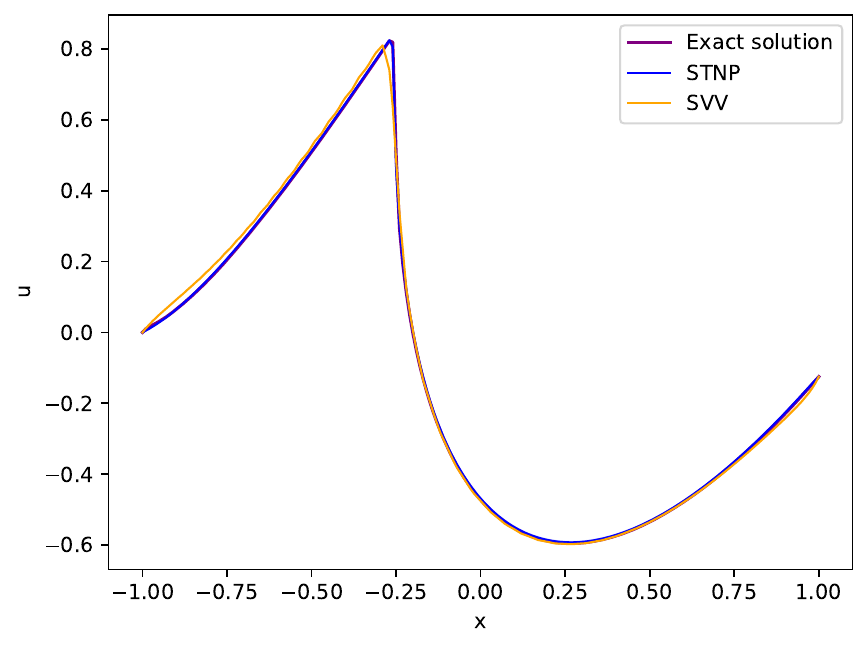}
    \end{overpic}\hfill
    \begin{overpic}[width = 0.45\textwidth]{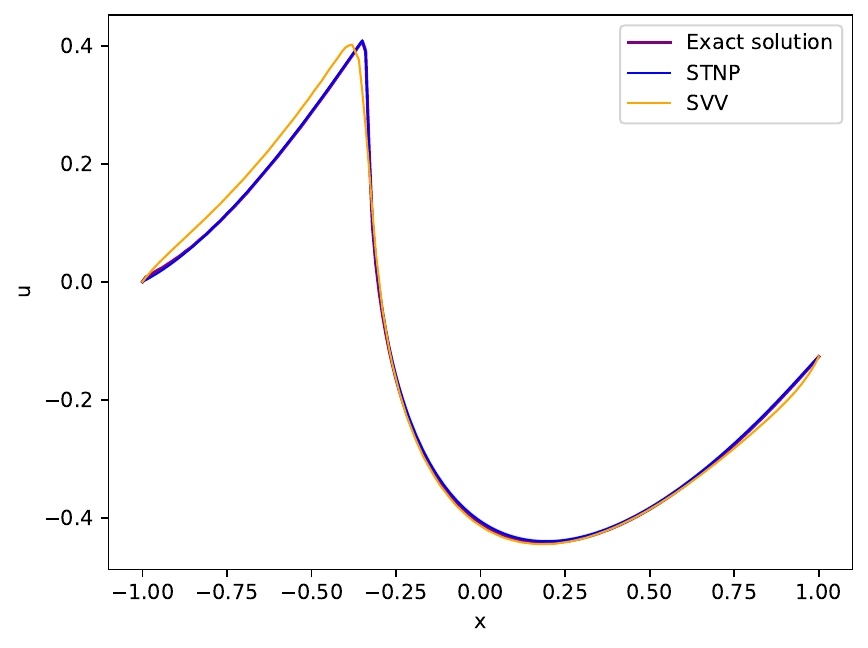}
    \end{overpic}

    \begin{overpic}[width = 0.45\textwidth]{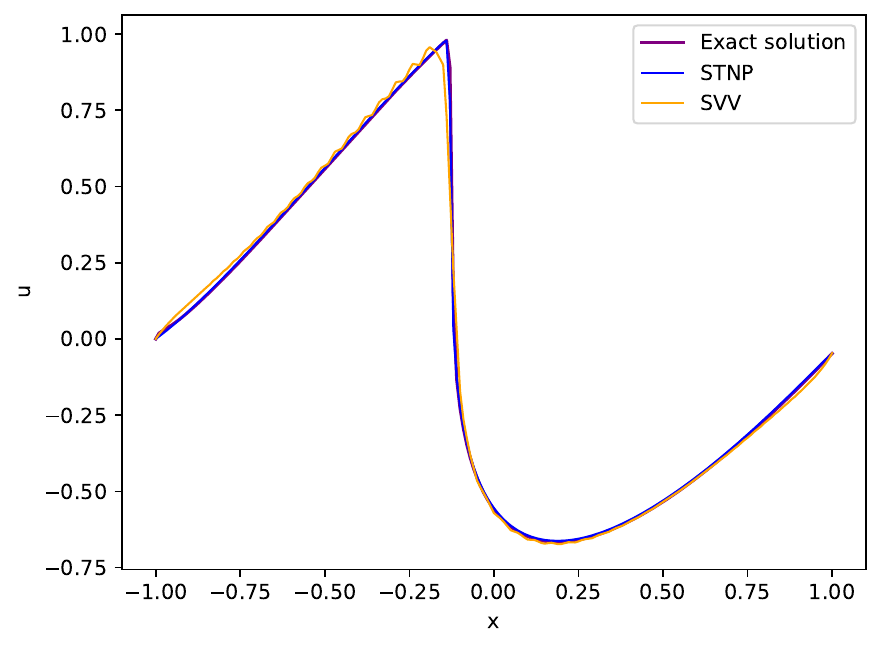}
    \end{overpic}\hfill
    \begin{overpic}[width = 0.45\textwidth]{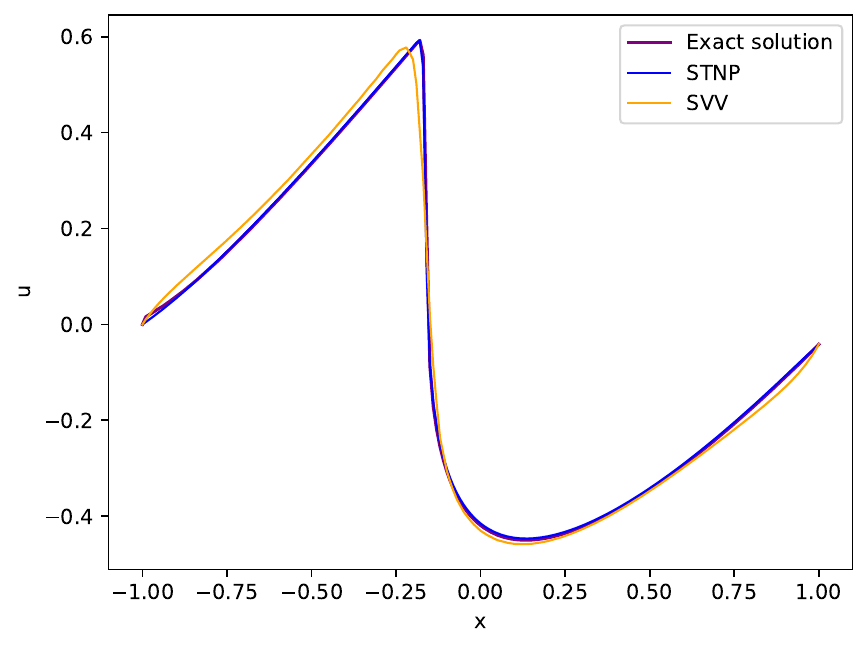}
    \end{overpic}
    \caption{Comparison of spectral-SVV, STNP, and exact solutions. The first column shows \(t=0.5\), the second column shows \(t=1.0\); the top row corresponds to \(\beta=0.6\) and the bottom row to \(\beta=0.8\). Both methods suppress spurious oscillations, but STNP achieves higher accuracy near the shock.}
    \label{FBENN_comparison}
\end{figure}

Next, we assess the impact of time integration on the STNP dynamics. We consider \(\beta=0.8\) and compare RK3 with an adaptive Runge--Kutta method (RK45), which adjusts the step size according to an embedded local truncation error estimate. As shown in the middle panel of Figure~\ref{Long_time_behavior}, RK45 achieves slightly smaller relative \(L^2\) error than fixed-step RK3. More importantly, RK45 provides more robust trajectories because it automatically reduces the step size when the solution  changes rapidly, which occurs around \(t\approx 0.3\) in this example. The right panel of Figure~\ref{Long_time_behavior} shows the cumulative time reached at each integration step and illustrates this adaptive step-size reduction. Notably, RK45 uses about \(451\) time steps on average, whereas RK3 always uses \(1000\), indicating that adaptive time stepping can improve both robustness and efficiency in our framework.

Finally, we investigate the long-time performance of STNP. We integrate the system up to \(t=5\) and repeat the experiment with ten independent random initializations. The resulting \(1\sigma\) confidence bands are narrow, indicating negligible sensitivity to initialization. The left panel of Figure~\ref{Long_time_behavior} reports the temporal evolution of the relative \(L^2\) error for STNP and spectral-SVV, showing that STNP maintains uniformly higher accuracy throughout the entire time horizon.

\begin{figure}[htbp]
    \centering
    \begin{overpic}[width=0.31\textwidth]{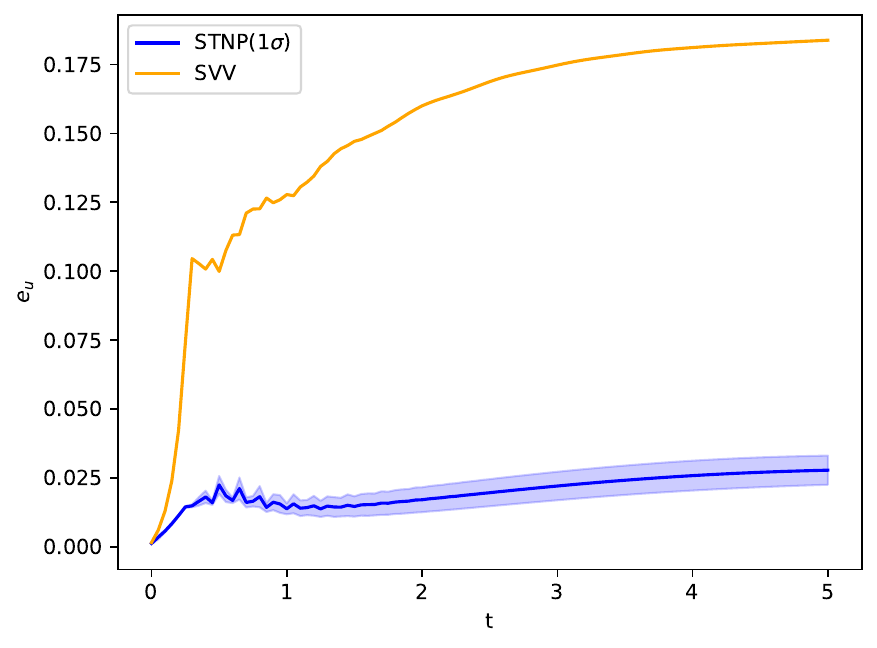}
    \end{overpic}\hfill
    \begin{overpic}[width=0.31\textwidth]{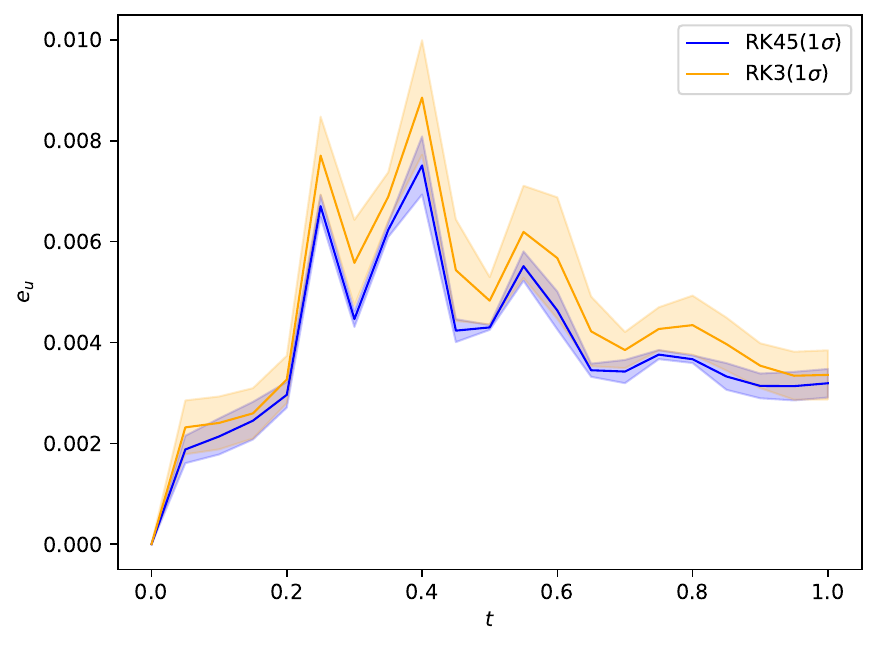}
    \end{overpic}\hfill
    \begin{overpic}[width=0.33\textwidth]{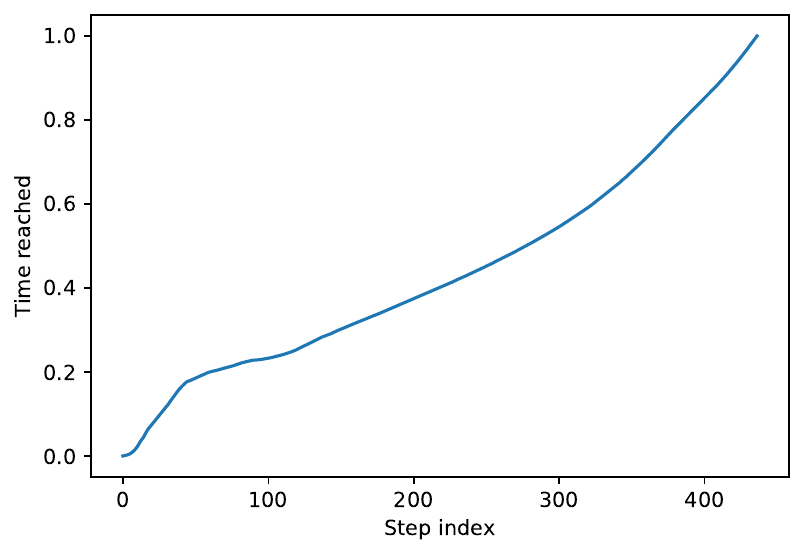}
    \end{overpic}
    \caption{Left: temporal evolution of the relative \(L^2\) error for STNP and spectral-SVV; STNP remains more accurate over the entire time horizon. Middle: relative \(L^2\) error of RK3 versus RK45, showing improved robustness with adaptive time stepping. Right: cumulative time reached at each integration step; RK45 automatically reduces the step size around \(t\approx 0.3\).}
    \label{Long_time_behavior}
\end{figure}

In summary, STNP consistently produces accurate, non-oscillatory solutions and outperforms the spectral-SVV baseline, especially near shocks. Moreover, it remains stable under long-time integration and shows negligible sensitivity to random initialization, indicating strong robustness in practice.

\section{Conclusion}
In this work, we introduced a sequential-in-time nonlinear parametrization (STNP) for fractional Burgers equations and studied its stability and accuracy. The method projects the PDE dynamics onto the tangent space of the neural manifold and computes the parameter evolution by solving a regularized least-squares problem at each step. This design yields a well-posed update and enables accurate tracking of the solution dynamics. We established stability and convergence estimates for both FBEFL and FBENN under the resulting projected dynamics, and the numerical results are consistent with the theory. In particular, STNP captures shocks without spurious oscillations, matches or outperforms WENO--Roe and spectral-SVV baselines with comparable parameter counts, and achieves \(10^{-8}\) relative error on smooth manufactured solutions while exhibiting the predicted convergence rates. Moreover, adaptive RK45 time integration improves robustness by automatically reducing step sizes near rapid transients in the parameter flow.

Several aspects of the framework are essential for future work. For inviscid hyperbolic conservation laws, adaptive collocation and residual-driven refinement near shocks appear crucial for accurately resolving the parameter dynamics. In addition, incorporating conservation or entropy constraints directly into the tangent-space projection remains an open direction. We expect these developments to further establish nonlinear parametrization as a practical tool for nonlocal, transport-dominated PDEs.

\section*{Acknowledgments}

This work was supported by the MURI grant (FA9550-20-1-0358), and the U.S. Department of Energy, Advanced Scientific Computing Research program, under the SEA-CROGS project, (DE-SC0023191).
Additional funding was provided by GPU Cluster for Neural PDEs and Neural Operators to support MURI Research and Beyond, under Award \#FA9550-23-1-0671.

\bibliographystyle{siamplain}
\bibliography{hyperref}

\end{document}